\documentclass[reqno]{amsart}                     

\usepackage{amsmath}
\usepackage{amssymb,amsfonts}
\usepackage{bbm}
\usepackage{xypic}
\xyoption{all}

\newtheorem{theorem}{Theorem}[section]

\newtheorem{corollary}[section]{Corollary}
\newtheorem{lemma}[section]{Lemma}
\newtheorem{remark}[section]{Remark}
\newtheorem{example}[section]{Example}
\newtheorem{fact}[section]{Fact}

\newcommand{\Ob}{\mathrm{Ob}}
\newcommand{\E}{\mathcal{M}}
\newcommand{\s}{\mathcal{S}}
\newcommand{\A}{\mathcal{A}}

\newcommand{\I}{\mathrm{I}}
\newcommand{\Aut}{\mathrm{Aut}}
\newcommand{\HH}{\mathbb{H}}

\title{On the third cohomology group of commutative monoids}

\author{M. Calvo-Cervera,   A.M. Cegarra,  B.A. Heredia}

\address{\newline
Departamento de \'Algebra, Facultad de Ciencias, Universidad de Granada.
\newline 18071 Granada, Spain \newline
 mariacc@ugr.es \ acegarra@ugr.es \ baheredia@ugr.es
                            }


\thanks{
The authors are much indebted to the referee,
whose useful observations greatly improved our exposition.
This work has been supported by `Direcci\'on General de
Investigaci\'on' of Spain, Project: MTM2011-22554, and for the first
and third authors also by FPU grants FPU12-01112 and  AP2010-3521,
respectively.
}

\begin{document}
\maketitle

\begin{abstract}
We interpret Grillet's symmetric third cohomology classes of
commutative monoids in terms of strictly symmetric  monoidal abelian
groupoids. We state and prove a classification result that
generalizes the well-known one for strictly commutative Picard
categories by Deligne, Fr\"ohlich and Wall, and Sinh.

\end{abstract}

\section{Introduction and summary}
The category  of commutative monoids is tripleable (monadic) over
the category of sets \cite{Mac1}, and so it is natural to specialize
Barr-Beck cotriple cohomology \cite{Barr-Beck} to define a
cohomology theory for commutative monoids. This was done in the 1990s by Grillet, to whose papers
\cite{grillet1,grillet3,grillet2} and book \cite{grillet4} we
refer the readers interested in cohomology theory for
commutative monoids (the topic of this paper). Although in
Subsection \ref{grillet} we review  the basic facts about the
resulting Grillet's cohomology, let us briefly recall here that, for
each commutative monoid $M$, its cohomology groups in this theory,
$H^n(M,\A)$, take coefficients in abelian group valued functors
 $\A$ on the category $\HH M$, whose
objects are the elements of $M$ and whose morphisms are pairs
$(a,b):a\to ab$, $a,b\in M$. Since these cohomology groups
$H^n(M,\A)$ can be computed, at less in low dimensions, by means of
{\em symmetric cochains}, they are usually referred as the {\em
symmetric cohomology groups} of the commutative monoid $M$.

For an arbitrary monoid $M$, that is, non necessarily commutative,
and $\A:\mathbb{D}M\to \mathbf{Ab}$ any abelian group valued functor
on the Leech category $\mathbb{D}M$ of morphisms $(a,b,c):b\to abc$,
 there are defined Leech's cohomology groups $H^n_{\mathrm
L}(M,\A)$ \cite{leech}. When the monoid $M$ is commutative, and
$\A:\HH M \to \mathbf{Ab}$ is any functor, then both cohomology
groups $H^n(M,\A)$ and $H^n_{\mathrm L}(M,\A)$ are defined, where
the coefficients for the Leech cohomology are here obtained by
composing $\A$ with canonical functor $\mathbb{D}M\to \HH M$,
$(a,b,c)\mapsto (b,ac)$. Although in dimension one we have that
 $H^1(M,\A)=H^1_{\mathrm L}(M,\A)$, in higher dimensions the
cohomology groups $H^n(M,\A)$ and $H^n_{\mathrm L}(M,\A)$ are,
however, different. Indeed, one easily argues that Leech cohomology
groups do not take properly account of the commutativity of the
monoid, in contrast to what happens with Grillet ones. Thus, for
example, while $H^2_{\mathrm L}(M,\A)$ classifies {\em all} group
coextensions of $M$ by $\A$ \cite[\S 2.4.9]{leech}, \cite[Theorem
2]{wells}, the symmetric two-dimensional cohomology group
$H^2(M,\A)$ classifies {\em commutative} group coextensions of $M$
by $\A$ \cite[\S V.4]{grillet4}.

In \cite{C-C-H}, we gave a natural interpretation for Leech
3-dimensional cohomology classes in terms of {\em monoidal abelian
groupoids}, that is, of small categories $\E$ all whose arrows are
invertible and whose isotropy groups $\mathrm{Aut}_\E(x)$ are all
abelian, endowed with a monoidal structure by a tensor product
$\otimes:\E\times \E\to \E$, a unit object $\I$, and coherent and
natural associativity and unit constraints $(x\otimes y)\otimes
z\cong x\otimes (y\otimes z)$, $x\otimes \I\cong x$, and $\I\otimes
x\cong x$ \cite{Mac2}. Thus, in \cite[Theorem 4.2]{C-C-H}, it is
stated that monoidal equivalence classes of monoidal abelian
groupoids are in one-to-one correspondence with isomorphism classes
of triples $(M,\A,k)$, consisting of a (non necessarily commutative)
monoid $M$, an abelian group valued functor $\A$ on the Leech
category $\mathbb{D}M$,  and a Leech 3-cohomology class $k\in
H^3_{\mathrm L}(M,\A)$.

In this paper, our goal is to state and prove a similar
interpretation for Grillet symmetric 3-cohomology classes, now in
terms of {\em strictly symmetric} (or {\em strictly commutative})
monoidal abelian groupoids \cite{deligne,Mac2,Saa}, that is,
monoidal abelian groupoids, as above, but now endowed with coherent
and natural isomorphisms $\boldsymbol{c}_{x,y}:x\otimes y\cong
y\otimes x$, satisfying the symmetry and strictness conditions
$\boldsymbol{c}_{y,x}\, \boldsymbol{c}_{x,y}=id_{x\otimes y}$ and
$\boldsymbol{c}_{x,x}=id_{x\otimes x}$. Our result here can be
summarized as follows (see Theorem \ref{mainthcl}):

\begin{itemize}{\em
\item Each symmetric $3$-cocycle $h\in Z^3(M,\A)$, of a commutative
monoid $M$ with coefficients in an abelian group valued functor
$\A:\HH M\to \mathbf{Ab}$, gives rise to a strictly symmetric
monoidal abelian groupoid $$ \s(M,\A,h).$$

\item For any strictly symmetric monoidal abelian groupoid  $\E$,
there exist a commutative monoid $M$, a functor $\A:\HH M\to
\mathbf{ Ab}$, a symmetric $3$-cocycle $h\in Z^3(M,\A)$, and a
symmetric monoidal equivalence
$$
\s(M,\A,h) \simeq \E.
$$

\item  For any two symmetric $3$-cocycles $h\in Z^3(M,\A)$ and
$h'\in Z^3(M',\A')$, there is a symmetric monoidal equivalence
$$
\s(M,\A,h)\simeq \s(M',\A',h')
$$
 if and and only if there exist an isomorphism of monoids
$i:M\cong M'$ and a natural isomorphism  $\psi:\A\cong \A'i$,
such that the equality of cohomology classes below holds.
$$
[h]=\psi_*^{-1}i^*[h']\in H^3(M,\A)
$$
}
\end{itemize}

Thus, triples $(M,\A, k)$, with $M$ a commutative monoid, $\A:\HH
M\to \mathbf{ Ab}$ a functor,  and $k\in H^3(M,\A)$ a symmetric
3-cohomology class, provide complete invariants for the
classification  of strictly symmetric monoidal abelian groupoids,
where two of them connected by a symmetric monoidal equivalence are
considered the same.

Our result particularizes to strictly commutative Picard categories
by giving, as a corollary, Deligne's well-known classification for
them \cite{deligne}, also proved independently by Fr\"{o}hlich and
Wall in \cite{Fr-Wa} and by Sinh in \cite{Sih,Sih2}. Indeed, in the
very special case where $M=G$ is an abelian group, any abelian group
valued functor on $\HH G$ is naturally equivalent to the constant
functor given by an abelian group $A$, and the symmetric
3-dimensional cohomology group $H^3(G,A)$ vanishes, whence Deligne's
result follows:  {\em Strictly commutative Picard categories are
classified by pairs $(G,A)$ of abelian groups}.

The organization of the paper is simple. After this introduction, it
contains two sections. The first is dedicated to stating a minimum
of necessary concepts and terminology, by reviewing some facts
concerning Grillet cohomology of commutative monoids (Subsection
\ref{grillet}) and symmetric monoidal groupoids (Subsection
\ref{symmon}). The second section comprises our classification
theorem for strictly symmetric monoidal abelian groupoids by means
of symmetric 3-cohomology classes.

\section{Preliminaries} This section aims to make this paper as
 self-contained as possible; hence, at the same
time as fixing notations and terminology, we also review some
necessary aspects and results about cohomology of commutative
monoids and symmetric monoidal categories that will be used
throughout the paper. However, the material in this preliminary
section is perfectly standard by now, so the expert reader may skip
most of it. For the cohomology  theory of commutative
monoids we mainly refer the reader to Grillet \cite[Chapters V, XII,
XIII, and XIV]{grillet4}, and for symmetric monoidal (= tensor)
categories to Mac Lane \cite{Mac2,Mac1} and Saavedra \cite{Saa}.

\subsection{Grillet cohomology of commutative monoids: Symmetric
cocycles.}\label{grillet}  Like most of cohomology theories in
Algebra, the cohomology of commutative monoids is a particular
instance of the cotriple cohomology by  Barr and Beck
\cite{Barr-Beck}. Briefly, let us recall that the category of commutative
monoids is tripleable over the category of sets and, for any given
commutative monoid $M$, the resulting cotriple
($\mathbb{G},\varepsilon ,\delta$) in the comma category
$\mathbf{CMon}\!\downarrow_M$, of commutative monoids over $M$, is
as follows. For each commutative monoid
 $X\overset{p}\to M$ over $M$, $$\mathbb{G}(X\stackrel{p }{\to
}M)=\mathbb{N}[X]\stackrel{\overline{p }}\to M,$$ where
$\mathbb{N}[X]$ is the free commutative monoid on the underlying set
$X$, and $\overline{p}$ is the homomorphism such that
$\overline{p}[x]=p(x)$ for any $x\in X$. The counit $\delta
:\mathbb{G}\to id$ sends $X\to M$ to the homomorphism in the comma
category $\delta:\mathbb{N}[X]\to X$ such that $\delta [x]=x$, and
the comultiplication $\varepsilon :\mathbb{G}\to \mathbb{G}^2$
carries each $X\to M$ to the homomorphism $\mathbb{N}[X]\to
\mathbb{N}[\mathbb{N}[X]]$ such that $\varepsilon [x]=[[x]]$, for
 $x\in X$. This cotriple produces a simplicial object
 $\mathbb{G}_\bullet$
in the category of endofunctors on $\mathbf{CMon}\!\downarrow_M$,
which is defined by $\mathbb{G}_n=\mathbb{G}^{n+1}$, with face and
degeneracy operators $d_i=\mathbb{G}^{n-i}\delta
\mathbb{G}^i:\mathbb{G}_n\to {\mathbb{G}}_{n-1}$ and
$s_i={\mathbb{G}^{n-i}\varepsilon
\mathbb{G}^i:}{\mathbb{G}}_{n}\to {\mathbb{G}}_{n+1},$ $0\leq i\leq
n$. Then, for any abelian group object  $\mathbf{A}$ in
$\mathbf{CMon}\!\downarrow_M$, one obtains a cosimplicial abelian
group $\mathrm{Hom}({\mathbb{G}}_{\bullet }(1_M),\mathbf{A})$, whose
associated cochain complex obtained by taking alternating sums of
the coface operators
$$\xymatrix {0\to
\mathrm{Hom}(\mathbb{G}(1_M),\mathbf{A})\stackrel{\partial ^0}{\to
}\mathrm{Hom}(\mathbb{G}^2(1_M),\mathbf{A})\stackrel{\partial ^1}{\to
}\mathrm{Hom}(\mathbb{G}^3(1_M),\mathbf{A})\overset{\partial
^2}\to\cdots \hspace{0.5cm} \big(\partial^n=
\sum\limits_{i=0}^{n+1}d_i^*\big)}$$ provides the {\em cotriple cohomology
groups of the commutative monoid $M$ with coefficients in
$\mathbf{A}$} by
$$
H^n_{\mathbb G}(M,\mathbf{A})=H^{n}(\mathrm{Hom}({\mathbb{G}}_{\bullet }(1_M),\mathbf{A})).
$$

In \cite{grillet1}, Grillet observes that, for any given commutative
monoid $M$, the category of abelian group objects in
$\mathbf{CMon}\downarrow_M$, is equivalent to the category of
abelian group valued functors $$\A:\HH M\to \mathbf{Ab},$$ where
$\HH M$ is the category with object set $M$ and arrow set $M\times
M$, where $(a,b):a\to ab$. Composition is given by
$(ab,c)(a,b)=(a,bc)$, and the identity of an object $a$ is $(a,1)$.
An abelian group valued functor, $\A:\HH M\to \mathbf{Ab}$, thus
consists of abelian groups $\A_a$, $a\in M$, and homomorphisms
$b_*:\A_a\to\A_{ab}$, $a,b\in M$, such that, for any $a,b,c\in M$,
$b_*c_*=(bc)_*:\A_a\to \A_{abc}$ and, for any $a\in M$,
$1_*=id_{\A_a}$. We refers to \cite[Chap. XXII, \S2]{grillet4} for
details but, briefly, let us say that the abelian group object
defined by an abelian group valued functor $\A:\HH M\to \mathbf{Ab}$
can be written as
$$E(M,\A)\to M,$$ where the {\em crossed
product} commutative monoid $E(M,\A)$ is the set $\bigcup_{a\in
M}\A_a\times\{a\}$  of all ordered pairs $(u_a,a)$ with
$a\in M$ and $u_a\in \A_a$, with multiplication given by
$$(u_a,a)(u_b,b)= (a_*u_b+ b_*u_a,ab).$$ The monoid homomorphism
$E(M,\A)\to\!M$ is the obvious projection
   $(u_a,a)\mapsto a$, and the internal group operation
$$E(M,\A)\times_M E(M,\A) \overset{+}\longrightarrow E(M,\A) $$
is defined by
   $(u_a,a)+(v_a,a)= (u_a+v_a,a)$.

Furthermore, in \cite{grillet1,grillet3,grillet2}, Grillet shows an
algebraically more lucid description of the low dimensional
cohomology groups
$$H^n(M,\A):=H^{n-1}_\mathbb{G}(M,E(M,\A)\to\!M)$$ by means of a
specific manageable complex (see also the recent work \cite{K-P})
\begin{equation}\label{sccomplex}
  0\to  C^1(M,\A)\overset{\partial}\longrightarrow  C^2(M,\A)\overset{\partial}
\longrightarrow
 C^3(M,\A)\overset{\partial}\longrightarrow  C^4(M,\A),
\end{equation}
called the complex of (normalized on $1\in M$) {\em symmetric
cochains} on $M$ with values in $\A$, which is defined as follows
(below, $\bigcup_{a\in M}\A_a$ is the disjoint union set of the
groups $\A_a$):

\vspace{0.2cm} $\bullet$ A {\em symmetric $1$-cochain} is a function
$f:M\to \bigcup_{a\in M}\A_a$, such that $f(a)\in \A_a$ and
$f(1)=0$.

\vspace{0.2cm} $\bullet$ A {\em symmetric $2$-cochain} is a function
$g:M^2\to \bigcup_{a\in M}\A_a$, such that $g(a,b)\in \A_{ab}$,
$$g(a,b)=g(b,a),$$ and $g(a,1)=0$.

\vspace{0.2cm} $\bullet$ A {\em symmetric $3$-cochain} is a function
$h:M^3\to \bigcup_{a\in M}\A_a$, such that $h(a,b,c)\in \A_{abc}$,
\begin{equation}\label{eqg3c}h(c,b,a)+h(a,b,c)=0,\hspace{0.3cm}
h(a,b,c)+h(b,c,a)+h(c,a,b)=0\end{equation} and $h(a,b,1)=0$.

\vspace{0.2cm} $\bullet$ A {\em symmetric $4$-cochain} is a function
$t:M^4\to \bigcup_{a\in M}\A_a$, such that $t(a,b,c,d)\in
\A_{abcd}$,
$$\begin{array}{l} t(a,b,b,a)=0,\hspace{0.4cm}t(d,c,b,a)+t(a,b,c,d)=0,\\
t(a,b,c,d)-t(b,c,d,a)+t(c,d,a,b)-t(d,a,b,c)=0,\\
t(a,b,c,d)-t(b,a,c,d)+t(b,c,a,d)-t(b,c,d,a)=0,
\end{array}$$ and
$t(a,b,c,1)=0$.

Under pointwise addition, these symmetric $n$-cochains constitute
the abelian groups $C^n(M,\A)$ in $(\ref{sccomplex})$, $1\leq n\leq
4$. The coboundary homomorphisms are defined by
\begin{flalign}\nonumber\bullet \hspace{0.2cm}&(\partial^1 f)(a,b)=
a_*f(b)-f(ab)+b_*f(a),&
\end{flalign}
\begin{flalign}\nonumber\bullet \hspace{0.2cm}&(\partial^2 g)(a,b,c)= a_*g(b,c)-g(ab,c)+g(a,bc)-c_*g(a,b),&
\end{flalign}
\begin{flalign}\label{eqg.3}\bullet \hspace{0.2cm}&(\partial^3 h)(a,b,c,d)= a_*h(b,c,d)-h(ab,c,d)+h(a,bc,d)-h(a,b,cd)+
d_*h(a,b,c).&
\end{flalign}

The groups
$$\begin{array}{l}Z^n(M,\mathcal{A})=\mathrm{Ker}\big(\partial^n: C^n(M,\mathcal{A})\to
 C^{n+1}(M,\mathcal{A})\big),\\[4pt]
B^n(M,\mathcal{A})=\mathrm{Im}\big(\partial^{n-1}:C^{n-1}(M,\mathcal{A})\to
 C^{n}(M,\mathcal{A})\big),
\end{array}
$$ are respectively  called  the groups of {\em
symmetric $n$-cocycles} and {\em symmetric $n$-coboundaries} on $M$
with values in $\mathcal{A}$. By \cite[Theorems 1.3 and
2.12]{grillet2},  there are natural isomorphisms
$$
H^n(M,\mathcal{A})\cong  Z^n(M,\mathcal{A})/ B^n(M,\mathcal{A})$$ for $n=1,2,3$.

\vspace{0.2cm}

The elements of $H^1(M,\A)=Z^1(M,\A)$ are {\em derivations} of $M$
in $\A$, that is, functions $f:M\to \bigcup_{a\in M}\A_a$ with
$f(a)\in \A_a$, such that $f(ab)=a_*f(b)+b_*f(a)$.

\vspace{0.2cm} The elements of $H^2(M,\A)$ have a natural
interpretation in terms of {\em commutative group coextension} of
the commutative monoid $M$ by the functor $\A:\HH M\to \mathbf{Ab}$.
This is the classification result by Grillet in \cite[\S
V.4]{grillet4}, whose proof is a good illustration of the one we give of
our result in this paper. We shall not present Grillet's proof here but,
briefly, let us recall that in the correspondence between symmetric
2-cohomology classes and isomorphism classes of commutative group
coextensions, each symmetric $2$-cocycle $g\in Z^2(M,\A)$ is carried
to the coextension
$$E(M,\A,g)\to M,$$
where the {\em twisted crossed product} commutative monoid
$E(M,\A,g)$ is  the set $\bigcup_{a\in M}\A_a\times\{a\}$ of all
 pairs $(u_a,a)$ with $a\in M$ and $u_a\in \A_a$, with
multiplication defined by
$$(u_a,a)(u_b,b)=(a_*u_b+b_*u_a+g(a,b),ab).
$$
This multiplication is unitary ($(0,1)$ is the unit) since $g$ is
normalized, that is,  $g(a,1)=0=g(1,a)$; and it is associative and
commutative due to $g$ being a symmetric 2-cocycle, that is, because
of the equalities $a_*g(b,c)+g(a,bc)=g(ab,c)+c_*g(a,b)$ and
$g(a,b)=g(b,a)$. The homomorphism $E(M,\A,g)\to M$ is the
projection $(u_a,a)\mapsto a$, and, for each $a\in M$, the simply
transitive group action of the group $\A_a$ on the fiber set over $a$
 is given by $$u_a\cdot (v_a,a)=(u_a+v_a, a).$$

\subsection{Strictly symmetric monoidal abelian groupoids}\label{symmon}
In order to fix some needed notations about those monoidal
categories we intend to address, we start by recalling that a {\em
groupoid}  is a small category all of whose morphisms are invertible. A
groupoid $\E$ whose isotropy (or vertex) groups $\Aut_\E(x)$,
$x\in\Ob\E$, are all abelian is termed an {\em abelian groupoid}
(cf. \cite[Definition 2.11.3 and Example 2.11.4]{B-B}, where the notion of abelian groupoid is
discussed under a categorical point of view). We will use additive
notation for abelian groupoids, thus, the identity morphism of an
object $x$ of an abelian groupoid $\E$ will be denoted by $0_x$; if
$u:x\to y$, $v:y\to z$ are morphisms, their composite is written
$v+u:x\to z$, while the inverse of $u$ is $-u:y\to x$.

\begin{example}\label{ka1}Any abelian group $A$ can be regarded as an abelian
groupoid $\E$ with only one object, say $a$, and $\Aut_\E(a)=A$. For
many purposes it is convenient to distinguish $A$ from the
one-object groupoid $\E$; the notation $(K(A,1),a)$ for $\E$ is not
bad (its nerve or classifying space \cite[I, Example 1.4]{Go-Ja} is
precisely the pointed Eilenberg-Mac Lane minimal complex $K(A,1)$
with base-vertex $a$), and we shall use it below.

A groupoid in which there is no morphisms between different objects
is called {\em totally disconnected}. It is easily seen that any
 totally disconnected abelian groupoid is actually a disjoint union
of abelian groups, or, more precisely, of the form $\bigcup_{a\in
M}(K(\A_a,1),a)$, for some family of abelian groups $(\A_a)_{a\in
M}$.
\end{example}

A {\em strictly symmetric} (or {\em strictly commutative}) {\em
monoidal abelian groupoid}
$$\E=(\E, \otimes, \I,
\boldsymbol{a},\boldsymbol{r},\boldsymbol{c})$$ consists of an
abelian groupoid $\E$, a functor $ \otimes: \E\times\E\to\E$ (the
{\em tensor product}), an object $\I$ (the {\em unit object}), and
natural isomorphisms $\boldsymbol{a}_{x,y,z}:(x\otimes y)\otimes z
\to x\otimes(y\otimes z)$, $\boldsymbol{r}_{x} :x\otimes \I \to x$,
and $\boldsymbol{c}_{x,y}: x\otimes y \to y\otimes x$ (called the
{\em associativity, unit, and symmetry constraints}, respectively),
such that the following five conditions are satisfied.
\begin{flalign} \label{eq1.1} \bullet \hspace{0.2cm}
&\boldsymbol{a}_{{x,y,z\otimes t}}+\boldsymbol{a}_{{x\otimes
y,z,t}}=(0_x\!\otimes\! \boldsymbol{a}_{{y,z,t}})+
\boldsymbol{a}_{{x,y\otimes z,t}}+
(\boldsymbol{a}_{{x,y,z}}\!\otimes\! 0_t),&
\end{flalign}
$$\xymatrix@C=18pt@R=18pt{((x\otimes y)\otimes z)\otimes t
\ar[r]^{\boldsymbol{a}}\ar[d]_{ \boldsymbol{a}\otimes 0}&(x\otimes y)\otimes (z\otimes t)
\ar[r]^{\boldsymbol{a}}& x\otimes (y\otimes (z\otimes t))
\\ (x\otimes (y\otimes z))\otimes t\ar[rr]^{\boldsymbol{a}}&&x\otimes ((y\otimes z)\otimes t)
\ar[u]_{0\otimes \boldsymbol{a}}}
$$
\begin{flalign}\label{eq1.2} \bullet \hspace{0.2cm}
&(0_y\!\otimes\!\boldsymbol{c}_{x,z})+\boldsymbol{a}_{y,x,z}+
(\boldsymbol{c}_{x,y}\!\otimes\!
0_z)=\boldsymbol{a}_{y,z,x}+\boldsymbol{c}_{x,y\otimes
z}+\boldsymbol{a}_{x,y,z},&
\end{flalign}
$$
\xymatrix@C=14pt@R=10pt{ &(y\otimes x)\otimes z\ar[r]^{\boldsymbol{a}}&
y\otimes (x\otimes z)\ar[rd]^{0\otimes \boldsymbol{c}} & \\
(x\otimes y)\otimes z\ar[ru]^{\boldsymbol{c}\otimes 0}
\ar[rd]^{\boldsymbol{a}}&& &y\otimes (z\otimes x) \\
&x\otimes( y\otimes z)\ar[r]^{\boldsymbol{c}}&
(y\otimes z)\otimes x\ar[ru]^{\boldsymbol{a}}&
}
$$
\begin{flalign}\label{eq1.3}\bullet \hspace{0.2cm}
 &(0_x\otimes \boldsymbol{r}_y)+
(0_x\otimes \boldsymbol{c}_{\I,y})
+\boldsymbol{a}_{x,\I,y}=\boldsymbol{r}_x\otimes 0_y,&
\end{flalign}
$$
\xymatrix@R=18pt{(x\otimes\I)\otimes y\ar[r]^{\boldsymbol{a}}\ar[d]_{\boldsymbol{r}\otimes  0}
&x\otimes (\I\otimes y)\ar[d]^{0\otimes \boldsymbol{c}}
\\ x\otimes y&x\otimes (y\otimes  \I)\ar[l]_{0\otimes  \boldsymbol{r}}
}
$$
\begin{flalign}\label{eq1.5}\bullet \hspace{0.2cm} &\boldsymbol{c}_{y,x}+\boldsymbol{c}_{x,y}=0_{x\otimes
y},&
\end{flalign}
$$
\xymatrix@C=5pt@R=18pt{x\otimes y\ar[rr]^{\boldsymbol{c}}\ar@{=}[dr]&&
y\otimes x\ar[dl]^{ \boldsymbol{c}}\\&x\otimes y&
}
$$
\begin{flalign}\label{eq1.6}\bullet \hspace{0.2cm} &\boldsymbol{c}_{x,x}=0_{x\otimes x}: x\otimes x\to x\otimes x.&
\end{flalign}

For later reference, let us remark that the conditions above imply
the following \cite{Ke}
\begin{flalign}\label{eq1.4} \bullet \hspace{0.2cm}
&(0_x\otimes\boldsymbol{r}_y)+\boldsymbol{a}_{x,y,\I}=\boldsymbol{r}_{x\otimes
y},&
\end{flalign}
$$
\xymatrix@C=-5pt@R=18pt{(x\otimes y)\otimes \I\ar[rr]^{\boldsymbol{a}}\ar[dr]_{ \boldsymbol{r}}&&
x\otimes(y\otimes\I)\ar[dl]^{0\otimes \boldsymbol{r}}\\&x\otimes y&
}
$$

These axioms guarantee the coherence of the constraints in the
following sense (see Mac Lane \cite[Theorem 5.1]{Mac2} and
Fr\"{o}hlich and Wall \cite[Theorem (5.2)]{Fr-Wa}).

\begin{fact}[Coherence Theorem] Let $\E$ be a strictly symmetric monoidal abelian groupoid. Then,
commutativity holds in every diagram  in  $\E$ with vertices
iterated instances of the functors $x\mapsto x$, the identity, $*\mapsto \I$, which selects the unit object, $(x,y)\mapsto x\otimes y$,
$(x,y)\mapsto y\otimes x$, and $x\mapsto
x\otimes x$, the diagonal functor, and whose edges are expanded instances of
$\boldsymbol{a}$, $\boldsymbol{a}^{-1}$, $\boldsymbol{c}$,
$\boldsymbol{r}$, and $\boldsymbol{r}^{-1}$.
\end{fact}

Below there is a convenient way to express this coherence in practice (see
Deligne \cite[1.4.1]{deligne} and Fr\"{o}hlich and Wall
\cite[Theorem (5.3)]{Fr-Wa}). Recall that, for any set $M$,  the free commutative monoid $\mathbb{N}[M]$ consists of commutative words
in $M$, which are unordered sequences $[a_1,\ldots,a_n]$ of
elements of $M$; unordered means that for any permutation $\sigma$,
$[a_{\sigma 1},\ldots,a_{\sigma n}]=[a_1,\ldots,a_n]$. Multiplication in $\mathbb{N}[M]$ is given by
by concatenation:
$$[a_1,\ldots,a_n][b_1,\ldots,b_m]=[a_1,\ldots,a_n,b_1,\ldots,b_m],$$
and the unit is $1=[~\,~]$, the empty word.

\begin{lemma}\label{fact1} Let $(x_a)_{a\in M}$ be any family
 of objects of a strictly symmetric monoidal abelian groupoid $\E$. If $\mathbb{N}[M]$ is the free commutative
 monoid generated by the index set $M$, then, there exists a map
 $F:\mathbb{N}[M]\to \Ob\E$ with $F[a]=x_a$, $a\in M$,  and
 isomorphisms $\varphi_{f,g}:Ff\otimes Fg\cong F(fg)$, $f,g\in
 \mathbb{N}[M]$, and $\varphi_0:\I\to F1$, satisfying the three equations below.

$\bullet$ \hspace{0.2cm} $ \varphi_{fg,h}+(\varphi_{f,g}\!\otimes\!
0_{Fh})=\varphi_{f,gh}+(0_{F\!f}\!\otimes\!
\varphi_{g,h})+\boldsymbol{a}_{F\!f,Fg,F\!h}$,
$$
\xymatrix@C=26pt@R=18pt{(Ff\otimes Fg)\otimes Fh\ar[r]^-{\varphi\otimes 0}
\ar[d]_{\boldsymbol{a}}&F(fg)\otimes Fh\ar[r]^-{\varphi}
&F(fgh)\ar@{=}[d]\\ Ff\otimes ( Fg\otimes Fh)\ar[r]^-{0\otimes \varphi}&
Ff\otimes F(gh)\ar[r]^-{\varphi}&F(fgh)}
$$

$\bullet$ \hspace{0.2cm} $\varphi_{g,f}
+\boldsymbol{c}_{F\!f,Fg}=\varphi_{f,g}$.
$$
\xymatrix@C=26pt@R=18pt{Ff\otimes Fg\ar[r]^-{\boldsymbol{c}}\ar[d]_{\varphi}
& Fg\otimes Ff\ar[d]^{\varphi}\\ F(fg)\ar@{=}[r]&F(gf)}
$$

$\bullet$ \hspace{0.2cm} $\varphi_{f,1}+ (0_{F\!f}\otimes
\varphi_0)= \boldsymbol{r}_{F\!f}$,
$$
\xymatrix@C=26pt@R=18pt{Ff\otimes \I\ar[r]^-{0\otimes \varphi_0}\ar[d]_{\boldsymbol{r}}
& Ff\otimes F1\ar[d]^{\varphi}\\ Ff\ar@{=}[r]&Ff}
$$
\end{lemma}
\begin{proof}
Let us now choose a total order for the index set $M$, so that any
$f\in \mathbb{N}[M]$ can be uniquely expressed as a sequence in
increasing order $$f=[a_1,\ldots,a_n], \hspace{0.4cm} a_1\leq
\cdots\leq a_n.$$ Then, we define $F:\mathbb{N}[M]\to \Ob\E$ by
putting $F1=\I$, $F[a]=x_a$,  and, recursively,
$$F[a_1,\ldots,a_n]=F[a_1,\ldots,a_{n-1}]\otimes x_{a_n}$$
for $n>1$. We have the identity isomorphism $\varphi_0=0_\I: \I\to
F1$ and, for any $f,g\in \mathbb{N}[M]$, it is clear that there is
an isomorphism
$$\varphi_{f,g}:Ff\otimes Fg\cong F(fg)$$
coming from instances of $\boldsymbol{a}$, $\boldsymbol{a}^{-1}$,
$\boldsymbol{c}$, and $\boldsymbol{r}$. It follows from the
Coherence Theorem above that these isomorphisms $\varphi_{f,g}$ so
obtained  satisfy all the requirements in the lemma.
\end{proof}

If $\E$, $\E'$ are strictly symmetric monoidal abelian groupoids,
then a {\em symmetric monoidal functor} $
F=(F,\varphi,\varphi_0):\E\to \E'$ consists of a functor on the
underlying groupoids $F:\E\to \E'$, natural isomorphisms $
\varphi_{x,y}:Fx\otimes  Fy\to F(x\otimes y)$, and an isomorphism
$\varphi_0:\I'\to F\I $, such that the following coherence
conditions hold.
\begin{flalign}\label{eq1.7}\bullet \hspace{0.2cm} &F\boldsymbol{a}_{x,y,z}+\varphi_{x\otimes
y,z}+(\varphi_{x,y}\!\otimes \! 0_{Fz})=\varphi_{x,y\otimes
z}+(0_{Fx}\!\otimes \! \varphi_{y,z})+\boldsymbol{a}'_{Fx,Fy,Fz},&
\end{flalign}
$$
\xymatrix@C=26pt@R=18pt{(Fx\otimes  Fy)\otimes  Fz\ar[r]^-{\varphi\otimes  0}
\ar[d]_{\boldsymbol{a}'}&F(x\otimes y)\otimes  Fz\ar[r]^-{\varphi}
&F((x\otimes y)\otimes z)\ar[d]^{F\boldsymbol{a}}\\ Fx\otimes  ( Fy\otimes  Fz)\ar[r]^-{0\otimes  \varphi}&
Fx\otimes  F(y\otimes z)\ar[r]^-{\varphi}&F(x\otimes (y\otimes z))}
$$
\begin{flalign}\label{eq1.8} \bullet \hspace{0.2cm} &F\boldsymbol{r}_x+\varphi_{x,\I}+(0_{Fx}\!\otimes \! \varphi_0)=
\boldsymbol{r}'_{Fx},&
\end{flalign}
$$
\xymatrix@C=26pt@R=18pt{Fx\otimes  \I'\ar[r]^-{0\otimes  \varphi_0}\ar[d]_{\boldsymbol{r}'}
& Fx\otimes  F\I\ar[d]^{\varphi}\\ Fx&F(x\otimes \I)\ar[l]_-{F\boldsymbol{r}}}
$$
\begin{flalign}\label{eq1.9}\bullet \hspace{0.2cm}  &\varphi_{y,x} +\boldsymbol{c}'_{Fx,Fy}=F\boldsymbol{c}_{x,y}+\varphi_{x,y}.&
\end{flalign}
$$
\xymatrix@C=26pt@R=18pt{Fx\otimes
Fy\ar[r]^-{\boldsymbol{c}'}\ar[d]_{\varphi} & Fy\otimes
Fx\ar[d]^{\varphi}\\ F(x\otimes
y)\ar[r]^{F\boldsymbol{c}}&F(y\otimes x)}
$$

Suppose $F':\E\to \E'$ is another symmetric monoidal functor. Then,
a {\em symmetric isomorphism} $\theta:F\Rightarrow F'$ is a natural
isomorphism on the underlying functors, $\theta:F\Rightarrow F'$,
such that the  coherence equations below are satisfied.
\begin{flalign}\nonumber\bullet \hspace{0.2cm}  &\theta_{x\otimes
y}+\varphi_{x,y}=\varphi'_{x,y}+(\theta_x\otimes \theta_y),
\hspace{0.3cm} \theta_\I+\varphi_0=\varphi'_0.&
\end{flalign}
$$
\xymatrix@C=26pt@R=18pt{Fx\otimes
Fy\ar[r]^-{\varphi}\ar[d]_{\theta\otimes \theta} & F(x\otimes
y)\ar[d]^{\theta}\\ F'x\otimes
F'y\ar[r]^{\varphi'}&F'(x\otimes y)}\hspace{0.45cm} \xymatrix@C=22pt@R=3pt{
&F\I\ar[dd]^{\theta}\\
\I'\ar[ru]^{\varphi_0}\ar[rd]_{\varphi'_0}&\\
&F'\I
}
$$

With compositions given in a natural way, strictly symmetric
monoidal abelian groupoids, symmetric monoidal functors, and
symmetric isomorphisms form a 2-category.
 A symmetric monoidal functor  $F:\E\to \E'$ is called a {\em
 symmetric monoidal
equivalence} if it is an equivalence in this 2-category, that is,
when there exists a symmetric monoidal functor $F':\E'\to \E$ and
symmetric isomorphisms $\theta: F'F\cong id_\E$ and
$\theta':FF'\cong id_{\E'}$.

 Our goal is to show a classification for strictly symmetric monoidal abelian groupoids,
 where two of them that are connected by a symmetric monoidal
 equivalence are considered the same. To do that, we will use the fact below by  Saavedra \cite[I,
4.4.5]{Saa}, where it is shown how to transport the symmetric monoidal
structure on an abelian groupoid along an equivalence on its
underlying groupoid. Recall that a functor between (non necessarily
abelian) groupoids $F:\E\to\E'$ is an equivalence (of categories) if
and only if the induced map on the sets of iso-classes of objects
$$
\Ob\E/_{\cong} \to \Ob\E'/_{\cong},\hspace{0.4cm} [x]\mapsto [Fx],
$$
is a bijection, and the induced homomorphisms on the automorphism
groups
$$
\Aut_\E(x)\to \Aut_{\E'}(Fx),\hspace{0.4cm} u\mapsto Fu
$$
are all isomorphisms \cite[Chapter 6, Corollary 2]{hig}.

\begin{fact}[Transport of Structure]\label{facttr} Let $F:\E\to \E'$ be an  equivalence
 between abelian groupoids, so
that there is a functor $F':\E'\to \E$ with natural equivalences
$\theta: id_\E\cong F'F$ and $\theta':FF'\cong id_{\E'}$ satisfying
$$
\theta'F+F\theta=id_F,\hspace{0.2cm} F'\theta'+\theta F'=id_{F'}.
$$

$(i)$ Any strictly symmetric monoidal structure on $\E$ can be
transported to one on $\E'$ such that the functors $F$ and $F'$
underlie symmetric monoidal functors, and the natural equivalences
$\theta$ and $\theta'$ turn to be symmetric isomorphisms.

\vspace{0.2cm}$(ii)$ If both $\E$ and $\E'$ have a strictly
symmetric monoidal structure, then any symmetric monoidal structure
on $F$ can be transported to one on $F'$ such that $\theta$ and
$\theta'$ become symmetric isomorphisms. Hence, a symmetric monoidal
functor is a symmetric monoidal equivalence if and only if the
underlying  functor is an equivalence.

\end{fact}

Concerning  Fact \ref{facttr}$(i)$, let us point out that for
any strictly symmetric monoidal structure $(\E, \otimes, \I,
\boldsymbol{a},\boldsymbol{r},\boldsymbol{c})$ on the abelian
groupoid $\E$,  the structure $(\E', \otimes , \I',
\boldsymbol{a}',\boldsymbol{r}',\boldsymbol{c}')$ transported onto
the abelian groupoid $\E'$ by means of $(F,F',\theta,\theta')$  is
such that the monoidal product $\otimes$ is the dotted functor in
the commutative square
$$
\xymatrix@R=18pt{\E'\times\E'\ar@{.>}[r]^-{\otimes }\ar[d]_{F'\times F'}&\E'\\
\E\times\E\ar[r]^-{\otimes}&\E,\ar[u]_{F}
}$$
and the unit object is $F\I$. The functors $F$ and $F'$ are endowed
with the isomorphisms
\begin{equation}\label{eqvarf}\varphi_{x,y}=-F(\theta_x\otimes\theta_y):Fx\otimes
Fy
\to F(x\otimes y),\hspace{0.3cm} \varphi_0=0_{FI}:F\I\to F\I,\end{equation}
$$\varphi'_{x',y'}=\theta_{F'x'\otimes F'y'}:F'x'\otimes
F'y'
\to F'(x'\otimes  y'),\hspace{0.3cm} \varphi'_0=\theta_\I:\I\to F'\I',
$$
and then the constraints $\boldsymbol{a}',\boldsymbol{r}'$ and the
symmetry $\boldsymbol{c}'$ are given by those isomorphisms uniquely
determined by the equations $(\ref{eq1.7})$, $(\ref{eq1.8})$, and
$(\ref{eq1.9})$, respectively.

Concerning  Fact \ref{facttr}$(ii)$, let us recall that if both
$\E$ and $\E'$ are strictly symmetric monoidal abelian groupoids,
then for any symmetric monoidal structure $(F,\varphi,\varphi_0)$ on
$F$, the structure $(F',\varphi',\varphi'_0)$ transported on the
functor $F'$ is such that the isomorphisms
$$\varphi'_{x',y'}:F'x'\otimes F'y'\to F'(x'\otimes y'), \hspace{0.4cm}\varphi'_0:\I\to F'\I',$$ are the uniquely determined by the dotted
arrows making commutative the diagrams below.
$$
\xymatrix@C=35pt@R=18pt{FF'x'\otimes
FF'y'\ar[d]_-{\varphi}\ar[r]^-{\theta'_{x'}\otimes \theta'_{y'}} & x'\otimes
y'\ar[d]^{-\theta'_{x'\otimes y'}}\\ F(F'x'\otimes
F'y')\ar@{.>}[r]^-{F\varphi'}&FF'(x'\otimes y')}\hspace{0.55cm} \xymatrix@C=22pt@R=3.5pt{
&F\I\ar@{.>}[dd]^-{F\varphi'_0}\\
\I'\ar[ru]^{\varphi_0}\ar[rd]_-{-\theta'_{\I'}}&\\
&FF'\I'
}
$$

\section{The classification theorem} The framework of our discussion below comes suggested by the known
classification theorems for strictly commutative Picard categories
given in \cite{deligne}, \cite{Fr-Wa} and \cite{Sih2}, for
categorical groups and Picard categories in \cite{Sih}, for braided
categorical groups in \cite{Jo-St}, for graded categorical groups,
braided graded categorical groups, graded Picard categories and
strictly commutative graded Picard categories in
\cite{C-G-O,cegarra2,cegarra3}, for braided fibred categorical
groups, fibred Picard categories and strictly commutative fibred
Picard categories in \cite{C-C}, and for monoidal groupoids in
\cite{C-C-H}.

Let $M$ be a commutative monoid and let $\A:\HH M\to \mathbf{Ab}$ be
a functor.  Each symmetric 3-cocycle $h\in Z^3(M,\A)$ gives rise to
a strictly symmetric monoidal abelian groupoid
\begin{equation}\label{gc} \s(M,\A,h)
\end{equation}
which should be thought of a sort of  {\em $2$-dimensional twisted
crossed product of $M$ by $\A$}, and it is built as follows: Its
underlying groupoid is the totally disconnected groupoid
$$\xymatrix{\bigcup_{a\in M}(K(\A_a,1),a),}$$ where, recall from Example \ref{ka1},  each
$(K(\A_a,1),a)$ denotes the groupoid having $a$ as its unique object
and $\A_a$ as the automorphism group of $a$. Thus, an object of $
\s(M,\A,h)$ is an element $a\in M$; if $a\neq b$ are different
elements of the monoid $M$, then there is no morphisms in
$\s(M,\A,h)$ between them, whereas its isotropy group at any $a\in
M$ is $\A_a$.

The tensor functor $$\otimes:\s(M,\A,h)\times \s(M,\A,h)\to
\s(M,\A,h)$$ is given on objects by multiplication in $M$, so
$a\otimes b=ab$, and on morphisms by the group homomorphisms
$$
\otimes:\A_a\times \A_b\to \A_{ab}, \hspace{0.4cm} u_a\otimes u_b=b_* u_a+ a_*u_b.
$$
The unit object is $\I=1$, the unit element of the monoid
$M$, and the structure constraints and the symmetry isomorphisms are
\begin{align}\nonumber
\boldsymbol{a}_{a,b,c}=\ &h(a,b,c):(ab)c\to a(bc),\\ \nonumber
\boldsymbol{c}_{a,b}=\ &0_{ab}:ab\to ba, \\ \nonumber
\boldsymbol{r}_{a}=\ &0_a:a1\to a,\nonumber
\end{align}
which are easily seen to be natural since $\A$ is an abelian group
valued functor.  The coherence condition $(\ref{eq1.1})$ holds
thanks to the cocycle condition $\partial^3h=0$ in \eqref{eqg.3},
while $(\ref{eq1.2})$ easily follows from the cochain equations in
\eqref{eqg3c}. The normalization condition $h(a,1,b)=0$, easily deduced from being $h(a,b,1)=0$, implies the
coherence condition $(\ref{eq1.3})$, and those in \eqref{eq1.5} and
\eqref{eq1.6} are obviously verified.

In the next Theorem \ref{mainthcl},  we observe how any strictly
symmetric monoidal abelian groupoid is symmetric monoidal equivalent
to  such  a 2-dimensional crossed product. Previously, we combine
the transport process in Fact \ref{facttr} with the generalization
of Brandt's Theorem \cite{Brandt}, which asserts that every groupoid
is equivalent as a category to a totally disconnected groupoid
\cite[Chapter 6, Theorem 2]{hig}, to obtain the following.

\begin{lemma}\label{redtdis} Any strictly symmetric monoidal abelian groupoid is
symmetric monoidal equivalent to one which is totally disconnected
and whose symmetry and unit constraints are all identities.
\end{lemma}
\begin{proof} Let $\E=(\E, \otimes, \I,
\boldsymbol{a},\boldsymbol{r},\boldsymbol{c})$ be any given strictly
symmetric monoidal abelian groupoid.

Let  $M=\Ob\E/_{\cong}$ be the set of isomorphism classes $[x]$ of
the objects of $\E$, and let us  choose, for each $a\in M$, any
representative object $x_a\in a$, with $x_{[\I]}=\I$.

In a first step, let us assume that all the symmetry constraints are
identities, that is, $x\otimes y=y\otimes x$ and
$\boldsymbol{c}_{x,y}=0_{x\otimes y}$, for any objects $x,y$ of
$\E$, and also that $\I\otimes\I=\I$ and $\boldsymbol{r}_\I=0_\I$,
the identity of the unit object. Then, let us form the totally
disconnected abelian groupoid
 $$\xymatrix{\E'=\bigcup_{a\in
M}(K(\A_a,1),a),}$$ whose set of objects is $M$, and whose isotropy group
at any object $a\in M$ is $\A_a=\Aut_\E(x_a)$.

This groupoid $\E'$ is  equivalent to the underlying groupoid $\E$.
To give a particular equivalence $F:\E\to \E'$, let us choose, for
each $a\in M$ and each $x\in a$, an isomorphism $\theta_x: x\cong
x_a$ in $\E$. In particular, for every $a\in M$, we take
  $\theta_{x_{\!a}\otimes \I}=\boldsymbol{r}_{x_{\!a}}$. Note that this selection implies that $\theta_\I=\theta_{\I\otimes\I}=\boldsymbol{r}_\I=0_\I$. Then, let
$F:\E\to \E'$ be the functor which acts on objects by $Fx=[x]$, and
on morphisms $u:x\to y$ by $Fu=\theta_y+u-\theta_x$. We have also
the more obvious functor $F':\E'\to \E$, which is defined on objects
by $F'a=x_a$, and on morphisms $u:a\to a$ by $F'u=u$. We have the
natural isomorphisms $\theta:id_\E\cong F'F$, and $\theta': FF'\cong
id_{\E'}$, where $\theta'_a=-\theta_{x_a}$, which clearly satisfy
the equalities $\theta'F+F\theta=id_F$ and $F'\theta'+\theta
F'=id_{F'}$.

Therefore, according to Fact \ref{facttr}, we can transport the
given symmetric monoidal structure of $\E$ to a corresponding one on
$\E'$ by means of $(F,F',\theta, id)$, so that we get a totally
disconnected strictly symmetric monoidal abelian groupoid $\E'=(\E',
\otimes , \I',
\boldsymbol{a}',\boldsymbol{r}',\boldsymbol{c}')$,
and a symmetric monoidal equivalence $F=(F,\varphi, \varphi_0):\E\to
\E'$. Now, a quick analysis of the structure on $\E'$ points out
that its unit object is $F\I=[\,\I\,]$ and that, for any object
$a\in \Ob\E'= M$,
\begin{align} \nonumber
\boldsymbol{r}'_a\ \overset{(\ref{eq1.8})}=\
&F(\boldsymbol{r}_{x_{\!a}})+\varphi_{{x_{\!a}},I}+
(0_a\otimes\varphi_0) \overset{\eqref{eqvarf}}=F(\boldsymbol{r}_{x_{\!a}})+\varphi_{{x_{\!a}},I}+(0_a\otimes 0_{[\I]})=
F(\boldsymbol{r}_{x_{\!a}})+\varphi_{{x_{\!a}},I}\\ \nonumber \overset{(\ref{eqvarf})}=\ &
\theta_{x_{\!a}}+\boldsymbol{r}_{x_{\!a}}-\theta_{x_{\!a}\otimes
\I}+\theta_{x_{\!a}\otimes \I}- (\theta_{x_{\!a}}\otimes
0_\I)-\theta_{x_{\!a}\otimes \I}\\
 \nonumber
 =\ &
\theta_{x_{\!a}}+\boldsymbol{r}_{x_{\!a}}- (\theta_{x_{\!a}}\otimes
0_\I)-\theta_{x_{\!a}\otimes \I} \overset{ \text{(naturality of $\boldsymbol{r}$)}}=
\boldsymbol{r}_{x_{\!a}}-\theta_{x_{\!a}\otimes \I}=0_{x_{\!a}}=0_a,
\end{align}
and, for any $a,b\in M$,
\begin{equation}\nonumber
\begin{split}
 \boldsymbol{c}'_{a,b} & \overset{\eqref{eq1.9}}{=} -\varphi_{x_b,x_a}+ F(\boldsymbol{c}_{x_a,x_b})
 +\varphi_{x_a,x_b}\overset{(\boldsymbol{c}=0)}=
  -\varphi_{x_b,x_a}+\varphi_{x_a,x_b}
 \\ & \overset{(\ref{eqvarf})}= \theta_{x_b\otimes x_a}-\theta_{x_b}\otimes \theta_{x_a}-\theta_{x_b\otimes x_a}+
\theta_{x_a\otimes x_b}+\theta_{x_a}\otimes\theta_{x_b}-\theta_{x_a\otimes x_b}
\\ &\text{(since $x_a\otimes x_b=x_b\otimes x_a$)}
\\ & =\theta_{x_b\otimes x_a}-\theta_{x_b}\otimes \theta_{x_a}
+\theta_{x_a}\otimes\theta_{x_b}-\theta_{x_a\otimes x_b}
\\ &\text{(since $\theta_{x_a}\otimes \theta_{x_b}=\theta_{x_b}\otimes \theta_{x_a}$, by the naturality
of  $\boldsymbol{c}_{x_a,x_b}=0_{x_a\otimes x_b}$)}
\\ &= \theta_{x_b\otimes x_a}-\theta_{x_a\otimes x_b} =  0_{x_a\otimes x_b}=0_{a b}.
  \end{split}
\end{equation}

Thus, $\E$ is symmetric monoidal equivalent to $\E'$, which  is a
totally disconnected strictly symmetric abelian groupoid whose unit
and symmetry constraints are all identities.

Hence, it suffices to prove now that the given strictly symmetric
monoidal
 abelian groupoid $\E$ is symmetric monoidal equivalent to another one whose
 symmetry constraints are all identities and whose unit constraint at the unit object is also the identity.
 Even more, following Deligne \cite{deligne}, we can prove that there is a symmetric
monoidal abelian groupoid $\mathcal{N}=(\mathcal{N},\bar{\otimes})$
whose constraints are all trivial (i.e., $\boldsymbol{a}=0$,
$\boldsymbol{c}=0$, and $\boldsymbol{r}=0$)  with a symmetric
monoidal equivalence $\mathcal{N}\simeq \E$:

Let $\mathbb{N}[M]$ be the free commutative monoid generated by $M$,
which we shall regard as a strictly symmetric monoidal discrete
groupoid (i.e., with only identities as morphisms). It follows from
Lemma \ref{fact1} that there is a symmetric monoidal functor
$$F=(F,\varphi,\varphi_0):\mathbb{N}[M]\to \E$$ such that
$F[a]=x_a$, for any $a\in M$. Then, we define $\mathcal{N}$ to be
the abelian groupoid whose set of objects is $\mathbb{N}[M]$, and
whose hom-sets are defined by
$$
\mathrm{Hom}_\mathcal{N}(f,g)=\mathrm{Hom}_\E(Ff,Fg).
$$
Composition in $\mathcal{N}$ is given by that in $\E$, so that we
have a full, faithful, and essentially surjective functor (i.e., an
equivalence)
$$
F:\mathcal{N}\to \E, \hspace{0.4cm} (f\overset{u}\to g)\mapsto (Ff\overset{u}\to Fg).
$$
The monoidal functor $\bar{\otimes}:\mathcal{N}\times \mathcal{N}\to
\mathcal{N}$ is defined by multiplication in $\mathbb{N}[M]$ on
objects, and on morphisms by
$$
(f\overset{u}\to g)\otimes (f'\overset{u'}\to g')= (ff'\overset{u\bar{\otimes} u'}\longrightarrow gg'),
$$
where ${u\bar{\otimes} u'}$ is the dotted morphism in the
commutative square in $\E$
\begin{equation}\label{bot}\begin{array}{l}
\xymatrix{Ff\otimes Ff'\ar[r]^{u\otimes u'}\ar[d]_{\varphi_{f,f'}}&Fg\otimes Fg'
\ar[d]^{\varphi_{g,g'}}\\ F(ff')\ar@{.>}[r]^{u\bar{\otimes} u'}&F(gg').
}
\end{array}
\end{equation}

So defined $\mathcal{N}=(\mathcal{N},\bar{\otimes})$ is a strictly
symmetric monoidal abelian groupoid with all the constraints being
identities. To prove this claim,  the following equalities on
morphisms in $\mathcal{N}$ should be verified
\begin{equation}\label{3ig}u\bar{\otimes}u'= u'\bar{\otimes}u,\hspace{0.25cm}
u\bar{\otimes}0_1=u, \hspace{0.25cm}
(u\bar{\otimes}u')\bar{\otimes}u''=u\bar{\otimes}(u'\bar{\otimes}u'').\end{equation}
But these follow from the naturality of the structure constraints of
$\E$, $\boldsymbol{c}$, $\boldsymbol{r}$, and $\boldsymbol{a}$,
respectively. For example, given any $u\in
\mathrm{Hom}_\mathcal{N}(f,g)$, we have the diagram
$$
\xymatrix@C=35pt{
Ff\otimes \I \ar@{}@<35pt>[d]|(.45){(A)}
\ar@/^2pc/[rr]^{\boldsymbol{r}_{Ff}}
\ar@{}@<15pt>[rr]|(.5){(C)}
\ar[r]^{0_{Ff}\otimes \varphi_0}\ar[d]_{u\otimes 0_\I}&Ff\otimes
F1\ar@{}@<40pt>[d]|(.45){(B)}\ar[r]^{\varphi_{f,1}}
\ar[d]^{u\otimes 0_{\!F1}}&Ff\ar[d]^u\\
Fg\otimes \I\ar[r]^{0_{Fg}\otimes \varphi_0}
\ar@{}@<-15pt>[rr]|(.5){(C)}
\ar@/_2pc/[rr]_{\boldsymbol{r}_{Fg}}
&Fg\otimes F1\ar[r]^{\varphi_{g,1}}&Fg }
$$
where the outside region commutes by naturality of $\boldsymbol{r}$,
those labelled with $(C)$ commute because
$(F,\varphi,\varphi_0):\mathbb{N}[M]\to \E$ is a symmetric monoidal
functor, and the square $(A)$ commutes due to $\otimes:\E\times
\E\to \E$ being a functor. It follows that the square $(B)$ is also
commutative and then that $u\bar{\otimes}0_1=u$. The other two
equations in \eqref{3ig} are proved similarly, and we leave them to
the reader.

Owing to the commutativity of squares \eqref{bot}, the isomorphisms
$\varphi_{f,f'}$ are natural  on morphisms of $\mathcal{N}$ and,
therefore, $F=(F,\varphi,\varphi_0):\mathcal{N}\to \E$ is actually a
symmetric monoidal functor, whence, by Fact \ref{facttr} $(ii)$, a
symmetric monoidal equivalence.
\end{proof}

We are now ready to prove the main result in this paper, namely

\begin{theorem}[Classification of Strictly Symmetric Monoidal Abelian Groupoids]\label{mainthcl}
 $(i)$ For any strictly symmetric monoidal abelian groupoid  $\E$, there
exist a commutative monoid $M$, a functor $\A:\HH M\to \mathbf{
Ab}$, a symmetric $3$-cocycle $h\in Z^3(M,\A)$, and a symmetric
monoidal equivalence
$$
 \s(M,\A,h) \simeq \E.
$$

$(ii)$ For any two commutative $3$-cocycles $h\in Z^3(M,\A)$ and
$h'\in  Z^3(M',\A')$, there is a symmetric monoidal equivalence
$$
 \s(M,\A,h) \simeq  \s(M',\A',h')
$$
 if and only if there exist an isomorphism of monoids $i:M\cong M'$ and a
natural isomorphism  $\psi:\A\cong \A'i$,  such that the equality of
cohomology classes below holds.
$$
[h]=\psi_*^{-1}i^*[h']\in H^3(M,\A)
$$
\end{theorem}
\begin{proof} $(i)$ By Lemma \ref{redtdis}, we can suppose that $\E$ is totally disconnected
and that all its symmetry and unit constraints are identities. In
assuming that hypothesis, let us write the underlying groupoid as
$\E=\bigcup_{a\in M}(K(\A_a,1),a)$, where $M=\Ob\E$ and, for each
$a\in M$, $\A_a=\Aut_\E(a)$. Then, a system of data $(M,\A,h)$, such
that $\s(M,\A,h) = \E$ as symmetric monoidal abelian groupoids, is
defined as follows:

- {\em The monoid $M$}. The function on objects of the tensor
functor $\otimes:\E\times \E\to\E$ determines a multiplication on
$M$, simply by putting $ab=a\otimes b$, for any $a,b\in M$. If we
write $1\in M$ for the the unit object of $\E$, then this
multiplication on $M$ is unitary, since the unit is strict.
Furthermore,   it is associative and commutative since, being $\E$
totally disconnected, the existence of the associativity and
symmetry constraints $(ab)c\to a(bc)$ and  $ab\to ba$ forces the
equalities $(ab)c=a(bc)$ and $ab=ba$. Thus, $M$ becomes a
commutative monoid.

- {\em The functor $\A:\HH M\to \mathbf{Ab}$}. The group
homomorphisms $\otimes:\A_a\times \A_b\to \A_{ab}$ have an
associative, commutative, and unitary behaviour, in the sense that
the equalities
\begin{equation}\label{acubeha}
(u_a\otimes u_b)\otimes u_c = u_a\otimes( u_b\otimes u_c),\hspace{0.2cm} u_a\otimes u_b=u_b\otimes u_a,
\hspace{0.2cm} 0_1\otimes u_a=u_a,
\end{equation}
hold. These follow from the abelianess of the groups of
automorphisms in $\E$, since  the  diagrams below commute due to the
naturality of the structure constraints.
$$
\xymatrix@R=26pt{(ab)c\ar[d]_{(u_a\otimes u_b)\otimes u_c}
\ar[r]^{\boldsymbol{a}_{a,b,c}}&a(bc)\ar[d]^{u_a\otimes (u_b\otimes u_c)}\\
(ab)c\ar[r]^{\boldsymbol{a}_{a,b,c}}&a(bc)
}\hspace{0.2cm}
\xymatrix@R=27pt{ab\ar[d]_{u_a\otimes u_b}
\ar[r]^{0_{ab}}&ba\ar[d]^{u_b\otimes u_a}\\
ab\ar[r]^{0_{ab}}&ba
}\hspace{0.2cm}
\xymatrix@R=26pt{a1=a\ar[d]_{u_a\otimes 0_1}
\ar[r]^-{0_a}&a\ar[d]^{ u_a}\\
a1=a\ar[r]^-{0_a}&a
}
$$

Then, if write $b_*:\A_a\to \A_{ab}$ for the homomorphism such that
$$b_*u_a:=0_b\otimes u_a=u_a\otimes 0_b,$$ the equalities
$$
\begin{array}{l} (bc)_*(u_a)=0_{bc}\otimes u_a=(0_b\otimes 0_c)\otimes u_a\overset{(\ref{acubeha})}=
0_b\otimes (0_c\otimes u_a )=b_*(c_*u_a), \\
1_*u_a=0_1\otimes u_a\overset{(\ref{acubeha})}=u_a,
\end{array}
$$
show that the assignments $a\mapsto \A_a$, $(a,b)\mapsto b_*:\A_a\to
\A_{ab}$, define an abelian group valued functor on $\HH M$. Observe
that this functor determines the monoidal product $\otimes$ of $\E$,
since
\begin{align}\nonumber u_a\otimes u_b\ =\ & (u_a+0_a)\otimes (0_b+
u_b) = (u_a\otimes 0_b)+ (0_a\otimes
u_b)\overset{(\ref{acubeha})}=(0_b\otimes u_a)+ (0_a\otimes
u_b)\\ =\ & b_*u_a+a_*u_b.
\nonumber
\end{align}

- {\em The symmetric $3$-cocycle $h\in  Z^3(M,\A)$}.  The
associativity constraints of $\E$ are necessarily written in the
form $\boldsymbol{a}_{a,b,c}=h(a,b,c)$, for some list
$\big(h(a,b,c)\in \A_{abc}\big)_{\!^{a,b,c\in M}}$. Since the
symmetry constraints are all identities, for any $(a,b,c)\in M^3$,
equation \eqref{eq1.2} gives
\begin{equation}\label{eau1}
h(a,b,c)+h(b,c,a)=h(b,a,c),
\end{equation}
which, making the permutation $(a,b,c)\leftrightarrow (b,a,c)$, is
written as $ h(b,a,c)+h(a,c,b)=h(a,b,c)$. If we carry this to
\eqref{eau1}, we obtain  $h(b,a,c)+h(a,c,b)+h(b,c,a)=h(b,a,c)$,
whence it follows the first symmetric cochain condition in
\eqref{eqg3c}, that is,
$$h(a,c,b)+h(b,c,a)=0.$$
To get the second one, we replace the term $h(b,a,c)$ with
$-h(c,a,b)$ in \eqref{eau1} and we have $$
h(a,b,c)+h(b,c,a)+h(c,a,b)=0,$$ as it is required.  Now, from
\eqref{eq1.4} it follows that
$h(a,b,1)=0$. Hence $h\in C^3(M,\A)$ is a
symmetric 3-cochain. Finally, the coherence condition in
\eqref{eq1.1} gives the equations
$$a_*h(b,c,d)+h(a,bc,d)+d_*h(a,b,c)=h(a,b,cd)+h(ab,c,d),$$
which means that $\partial^3h=0$ in \eqref{eqg.3}, so that $h\in
Z^3(M,\A)$ is  a symmetric 3-cocycle.

Since an easy comparison shows that $\E= \s(M,\A,h)$, the proof of
this part is complete.

\vspace{0.2cm} $(ii)$ We first assume that there exist an
isomorphism of monoids $i:M\cong M'$ and a natural isomorphism
$\psi:\A\cong \A'i$,  such that $ \psi_*[h]=i^*[h']\in H^3(M,\A'i)
$. This  means that there is a symmetric 2-cochain $g\in
C^2(M,\A'i)$ such that the equalities below hold.
\begin{align} \label{qa1}
  \psi_{abc} h(a,b,c)=& h'(ia,ib,ic)+(ia)_*g(b,c)-g(ab,c)+g(a,bc)-(ic)_*g(a,b)
 \end{align}
Then, we have a symmetric monoidal isomorphism
\begin{equation}\label{sf}\s(i,\psi,g)=(F,\varphi,\varphi_0): \s(M,\A,h)\to
 \s(M',\A',h'),\end{equation} whose underlying functor acts by
$$F(a\overset{u}\to a)=(ia\overset{\psi_au_a}\longrightarrow ia),$$
and whose structure isomorphisms are given by
$$
\begin{array}{l} \varphi_{a,b}=g(a,b):(ia)\, (ib) \to i(ab),\\[4pt]
\varphi_0=0_1:1\to i1=1.
\end{array}
$$
In effect, so defined, it is easy to see that $F$ is an isomorphism
between the underlying  groupoids. Verifying the naturality of the
isomorphisms $\varphi_{a,b}$, that is, the commutativity of the
squares
\begin{equation}\label{natfi}\begin{array}{l}
\xymatrix{(ia)(ib)\ar[r]^{\varphi_{a,b}}\ar[d]_{(ia)_*\psi_{b}(u_b)+
(ib)_*\psi_{a}(u_a)}
&i(ab)\ar[d]^{\psi_{ab}(a_*u_b+b_*u_a)}\\ (ia)(ib)\ar[r]^{\varphi_{a,b}}&i(ab),
}\end{array}
\end{equation}
for  $u_a\in \A_a$, $u_b\in\A_b$, is equivalent (since the groups
$\A'_{i(ab)}$ are abelian) to verify the equalities
\begin{equation}\label{natfi2}\psi_{ab}(a_*u_b+b_*u_a)=(ia)_*\psi_b(u_b)+(ib)_*\psi_a(u_a),
\end{equation}
which hold since the naturality of $\psi:\A\cong \A' i$ just says
that
\begin{equation}\label{ea3}\psi_{ab}(a_*u_b)=(ia)_*\psi_b(u_b).\end{equation}
 The coherence
condition $(\ref{eq1.7})$ is verified as follows
\begin{align} \label{aeq5}
   &\varphi_{a,b\otimes c}+(0_{Fa}\!\otimes \! \varphi_{b,c})+\boldsymbol{a}'_{Fa,Fb,Fc}=
   \varphi_{a,bc}+ (ia)_*\varphi_{b,c}+h'(ia,ib,ic) \\ \nonumber
   &\hspace{0.3cm}=g(a,bc)+ (ia)_*g(b,c)+h'(ia,ib,ic)
\overset{(\ref{qa1})}= \psi_{abc}h(a,b,c)+g(ab,c)+(ic)_*g(a,b)\\[4pt] \nonumber
&\hspace{0.3cm}=\psi_{abc}h(a,b,c)+\varphi_{ab,c}+i(c)_*\varphi_{a,b}
   = F(\boldsymbol{a}_{a,b,c})+\varphi_{a\otimes
b,c}+(\varphi_{a,b}\!\otimes \! 0_{Fc}),
\end{align}
whilst the conditions in  $(\ref{eq1.8})$  and $(\ref{eq1.9})$
trivially follow from the symmetric cochain conditions
$g(a,1)=0_{ia}$ and  $g(a,b)=g(b,a)$, respectively.

\vspace{0.2cm} Conversely, suppose that
$$F=(F,\varphi,\varphi_0): \s(M,\A,h) \to  \s(M',\A',h')$$ is
any symmetric monoidal equivalence. By \cite[Lemma 3.1]{cegarra3},
there is no loss of generality in assuming that $F$ is strictly
unitary in the sense that $\varphi_0=0_1:1\to 1=F1$.

As the underlying  functor establishes an equivalence between the
underlying  groupoids, $$\xymatrix{F:\bigcup_{a\in
M}(K(\A_a,1),a)\simeq \bigcup_{a'\in M'}(K(\A'_{a'},1),a'),}$$ and
these are totally disconnected, it is necessarily an isomorphism.
Let us write $i:M\cong M'$ for the bijection describing the action
of $F$ on objects; that is, such that $ia=Fa$, for each $a\in M$.
Then, $i$ is actually an isomorphism of monoids, since the existence
of the structure isomorphisms $\varphi_{a,b}: (ia)(ib)\to i(ab)$
forces the equality $(ia)(ib)=i(ab)$.

Let us write $\psi_a:\A_a \cong \A'_{ia}$ for the isomorphism giving
the action of $F$ on automorphisms $u_a:a\to a$, that is, such that
$\psi_a u_a=Fu_a$, for each $u_a\in \A_a$, and  $a\in M$. The
naturality of the automorphisms $\varphi_{a,b}$ tell us that the
equalities $(\ref{natfi2})$ hold (see diagram $(\ref{natfi})$).
These, for the case when $u_a=0_a$, give the equalities in
$(\ref{ea3})$,  which amounts to saying that $\psi:\A\cong \A'i$ is
a natural isomorphism of abelian group valued functors on $\HH M$.

Writing now $g(a,b)=\varphi_{a,b}$, for each $a,b\in M$,  the
equations $g(a,1)=0_{ia}$ and $g(a,b)=g(b,a)$  hold just due to the
coherence equations $(\ref{eq1.8})$  and $(\ref{eq1.9})$, and thus
we have a symmetric 2-cochain $g=\big(g(a,b)\in
\A'_{i(ab)}\big)_{^{a,b\in M}}$, which satisfies the equations
$(\ref{qa1})$ owing to the coherence equations $(\ref{eq1.7})$, as
we can see just by retracting our steps in $(\ref{aeq5})$. This
means that $\psi_*(h) = i^*(h')+\partial^2(g)$ and, therefore, we
have that $\psi_*[h] = i^*[h']\in  H^3(M,\A' i)$, whence  $[h] =
\psi_*^{-1}i^*[h']\in H^3(M,\A)$, as required.
\end{proof}

\begin{remark}\label{r11} Let $$\text{\bf Symmetric 3-cocycles}$$ denote the {\em category
of $3$-cocycles of commutative monoids}. That is, the category whose
objects are triples $(M,\A,h)$ with $M$ a commutative monoid,
$\A:\HH M\to \mathbf{Ab}$ a functor, and $h\in Z^3(M,\A)$ a
symmetric 3-cocycle, and whose arrows $$ (i,\psi,[g]):(M,\A,h) \to
(M',\A',h') $$ are triples consisting of a monoid homomorphism
$i:M\to M'$, a natural transformation $\psi:\A\to \A'i$, and the
equivalence class $[g]$ of a symmetric 2-cochain $g\in C^2(M,\A'i)$
such that $\psi_*(h)=i^*(h')+\partial^2(g)$ (i.e., equation
\eqref{qa1} holds). Two such cochains $g,g'\in C^2(M,\A'i)$ are
equivalent if there is a symmetric 1-cochain $f\in C^1(M,\A'i)$ such
that $g=g'+\partial^1(f)$. Composition in this category of
3-cocycles is defined in a natural way: The composite of
$(i,\psi,[g])$ with $ (i',\psi',[g']):(M',\A',h') \to
(M'',\A'',h'')$ is the arrow
$$
(i'i,\psi'i\ \psi, [(\psi'i)_*(g)+i^*(g')]):(M,\A,h) \to
(M'',\A'',h''),
$$
where $i'i:M\to M''$ is the composite homomorphism of $i'$ and $i$,
$\psi'i\ \psi:\A\to \A''i'i$ is the natural transformation such that
$(\psi'i\ \psi)_a= \psi'_{ia} \psi_a$, the composite
homomorphism of $\psi'_{ia}:\A'_{ia}\to \A''_{i'ia}$ with
$\psi_a:\A_a\to \A'_{ia}$, for each $a\in M$, and
$(\psi'i)_*(g)+i^*(g')\in C^2(M,\A''i'i)$ is the symmetric 2-cochain
given by
$$
((\psi'i)_*(g)+i^*(g'))(a,b)=\psi'_{i(ab)}g(a,b)+g'(ia,ib).
$$
The identity arrow of any object $(M,\A,h)$ is the triple
$(id_M,id_{\A},[0])$.

With a slight adaptation of the arguments in the proof of part (ii),
Theorem 3.1 can be formulated as an equivalence of categories
$$
\text{\bf Symmetric 3-cocycles} \simeq \text{\bf Strictly symmetric monoidal abelian groupoids}
$$
between the category of symmetric 3-cocycles and the category of
strictly symmetric monoidal abelian groupoids, $\E$, with
iso-classes, $[F]:\E\to \E'$, of symmetric monoidal functors,
$F:\E\to \E'$, as arrows. The equivalence of categories is given by
the constructions
 \eqref{gc} on objects and \eqref{sf} on morphisms, that is,
$$
\xymatrix@C=40pt{\big((M,\A,h)\ar[r]^{(i,\psi,[g])}&(M',\A',h')\big)}\mapsto
\xymatrix@C=40pt{\big(\s(M,\A,h)\ar[r]^{[\s(i,\psi,g)]}&\s(M',\A',h')\big).}
$$
\end{remark}

A {\em strictly commutative Picard category} \cite[Definition
1.4.2]{deligne} is a strictly symmetric monoidal abelian groupoid
$\mathcal{P}=(\mathcal{P}, \otimes, \I,
\boldsymbol{a},\boldsymbol{r},\boldsymbol{c})$ in which, for any
object $x$, there is an object $x^*$ with an arrow $x\otimes x^*\to
\I$. Actually, the hypothesis of being abelian is superfluous here
since a monoidal groupoid in which every object has a quasi-inverse
is always abelian \cite[Proposition 4.1 (ii)]{C-C-H}. Next, we
obtain Deligne's classification result for these Picard categories
as a corollary of Theorem \ref{mainthcl} and the lemma below by Mac
Lane \cite[Theorem 4]{MacL}.

\begin{lemma}\label{MacLane} Let $G$ be any abelian group. For any abelian group $A$, regarded as a constant functor $A:\HH G\to \mathbf{Ab}$, the symmetric $3$-dimensional cohomology group of $G$ with coefficients in $A$ is zero, that is, $H^3(G,A)=0$.
\end{lemma}

For any abelian groups $G$ and $A$, let $\s(G,A,0)$ be the strictly
symmetric monoidal abelian groupoid  built as in $\eqref{gc}$, for
the constant functor $A:\HH G\to \mathbf{Ab}$ and the zero 3-cocycle
$0:G^3\to A$.
 Since $G$ is a group, $\s(G,A,0)$ is actually  a
strictly commutative Picard category. Then, we have

\begin{corollary}[Deligne \cite{deligne}, Fr\"{o}hlich-Wall \cite{Fr-Wa}, Sinh \cite{Sih2}]\label{clbcg}

$(i)$ For any strictly commutative Picard category $\mathcal{P}$,
there exist abelian groups $G$ and $A$ and a symmetric monoidal
equivalence
$$
\s(G,A,0) \simeq \mathcal{P}.
$$

$(ii)$ For any abelian groups  $G,G',A$ and $A'$, there is a
symmetric monoidal equivalence
$$
\s(G,A,0)\simeq \s(G',A',0)
$$
 if and and
only if there are isomorphisms $G\cong G'$ and $A\cong A'$.
\end{corollary}
\begin{proof}
$(i)$ Let $\mathcal{P}$ be a strictly commutative Picard category.
By Theorem \ref{mainthcl}, there are a commutative monoid $M$, a
functor $\A:\HH M\to \mathbf{ Ab}$,  a 3-cocycle $h\in Z^3(M,\A)$,
and a symmetric monoidal equivalence $ \s(M,\A,h) \simeq
\mathcal{P}$.

Then, $\s(M,\A,h)$ is a strictly commutative Picard category as
$\mathcal{P}$ is and, therefore, for any $a\in M$, it must exist
another $a^*\in M$ with a morphism $a\otimes a^*=aa^*\to\I= 1$ in
$\s(M,\A,h)$. Since the groupoid $\s(M,\A,h)$ is totally
disconnected, it must be
 $aa^*=1$ in $M$, which  means that $a^*=a^{-1}$ is an inverse of $a$ in $M$. Therefore,
$M=G$ is actually an abelian group.

Let $A=A_1$ be the abelian group attached by $\A$ at the unit of
$G$. Then, a natural isomorphism $\phi:A\cong \A$ is defined such
that, for any $a\in G$,  $\phi_a=a_*:A=A_1\to\A_a$. Therefore,
Theorem \ref{mainthcl}$(ii)$  and Lemma \ref{MacLane} give the
existence of a symmetric monoidal equivalence $$ \s(G,\A,h)\simeq
\s(G,A,0),$$ whence a symmetric monoidal equivalence
$\s(G,A,0)\simeq \mathcal{P}$ follows.

$(ii)$ This follows directly form Theorem \ref{mainthcl}$(ii)$.
\end{proof}

\begin{remark} As in Remark \ref{r11}, the classification result above can be formulated in terms of an equivalence between the category of strictly commutative Picard
categories, with iso-classes of symmetric monoidal functors as morphisms, and the category of pairs $(G,A)$ of abelian groups, with morphisms
$$(i,\psi, k):(G,A)\to (G',A')$$
triples consisting of two group homomorphisms $i:G\to G'$, $\psi:A\to A'$, and a cohomology class $k\in H^2(G,A')=\mathrm{Ext}_{\mathbb{Z}}(G,A')$, where composition is given by $$(i',\psi', k')(i,\psi, k)=(i'i,\psi'\psi, \psi'_*(k)+i^*(k')).$$

\end{remark}

\end{document}